\documentclass[12pt]{article}
\usepackage[mathscr]{eucal}
\usepackage{amssymb,latexsym}
\usepackage{verbatim}
\usepackage{amsmath}
\usepackage{amsthm}
\usepackage{pdfsync}

\usepackage[normalem]{ulem} 
\usepackage{color}
\usepackage{authblk}

\usepackage[%
bookmarks=true,
colorlinks,
linkcolor=blue,
urlcolor=blue,
citecolor=blue,
plainpages=false,
pdfpagelabels,
final,
breaklinks=true\between
]{hyperref}
\usepackage{hyperref}

\newtheorem{thm}{Theorem}[section]
\newtheorem{lem}[thm]{Lemma}

\newcommand\calE{{\mathcal{E}}}

\DeclareMathOperator{\inn}{in}
\DeclareMathOperator{\out}{out}

\renewcommand\l{\lambda}

\newcommand\bbR{{\mathbb R}}
\newcommand\bbZ{{\mathbb{Z}}}

\newcommand\vep{\varepsilon}

\renewcommand\S{\Sigma}

\renewcommand\d{\partial}

\newcommand\D{\nabla}
\newcommand\e{\epsilon}

\renewcommand\div{{\rm div}}

\renewcommand\l{\lambda}

\newcommand\8{\infty}

\renewcommand\th{\theta}

\newcommand\vs{\vspace}

\newcommand\beq{\begin{eqnarray}}
\newcommand\eeq{\end{eqnarray}}
\newcommand\ben{\begin{enumerate}}
\newcommand\een{\end{enumerate}}
\newcommand\bit{\begin{itemize}}
\newcommand\eit{\end{itemize}}

\newcounter{mnotecount}[section]

\title{Positive mass theorems for manifolds with \\ asymptotically locally hyperbolic toroidal ends}

\author[1]{Gregory J. Galloway\footnote{galloway@math.miami.edu}}
\author[2]{Tin-Yau Tsang\footnote{tytsang@math.ubc.ca}}

\affil[1]{Department of Mathematics,
\linebreak University of Miami, Coral Gables, FL, USA}

\affil[2]{Department of Mathematics, 
\linebreak University of British Columbia, Vancouver, BC, CA}

\begin{document}

\date{}

\maketitle

\begin{abstract}  
In [Classical Quantum Gravity {\bf 35}, 115015 (2018)], P. Chru\'sciel, L. Nguyen, T.-T. Paetz and the 
first-named author obtained a positive mass theorem for asymptotically locally hyperbolic manifolds with boundary, having a toroidal end. The proof made use of properties of marginally outer trapped surfaces (MOTS).  Here we present some new positive mass results for such manifolds, but without boundary, which allow for other more general ends.  The proofs, while still  MOTS-based, involve a more elaborate technique (related to $\mu$-bubbles) introduced in work of D. A. Lee, M. Lesourd, and R. Unger 
[Calculus Var. Part. Differ. Equations {\bf 62}(7), 194 (2023)] for manifolds with an asymptotically flat end, and further developed by the second-named author in [arXiv:2604.26978 (2026)]  for manifolds with an asymptotically hyperbolic end.  
\end{abstract}

\section{Introduction}
In \cite{ChruDelay}, Chru\'sciel and Delay established positivity of mass in the asymptotically hyperbolic (AH) setting, without requiring a spin assumption.  These results have recently been  extended to the case with ``arbitrary ends" \cite{TinYau}.  
A result in the asymptotically locally hyperbolic (ALH) setting  was considered in \cite{CGNP}.
Theorem~1.1 in \cite{CGNP} establishes the nonnegativity of the mass of a toroidal end 
$\calE$  in an ALH manifold $(M^n,g)$, $4 \le n \le 7$, assuming  the scalar curvature  satisfies $R \ge -n(n-1)$, and the mean curvature of $\d\calE$ satisfies $H < n-1$. 
In fact, as noted in \cite{ChruPRD}, it is sufficient to assume $H \le n-1$. Moreover, the assumed product topology can be weakened; see \cite[Remark~6]{ChruPMT}.  

Theorem 1.1 in \cite{CGNP} also considers the case of an ALH manifold with boundary, with a spherical space (quotient of a sphere) end, the proof of which is 
substantially different from the torus case.  The spherical space case with ``arbitrary ends'' is addressed in~\cite{TinYau}.  

In this note we  consider the ALH case with a toroidal end, which otherwise allows for more general ends. 
We focus here on proving the following.
    
 \begin{thm}\label{thm:pm1}  Let $(M^n,g)$, $4 \le n \le 7$, be a complete orientable ALH manifold with a conformally compactifiable toroidal end $\calE$, 
and with scalar curvature $R \ge -n(n-1)$.  Suppose further that $N = \overline{M\setminus\calE}$ is noncompact and asymptotically retractible onto $\d\cal{E}$. Then the mass of the end $\calE$ is nonnegative. 
\end{thm}   

Various terms in the statement will be defined in Section 2.  The `asymptotically retractible' assumption is trivially satisfied if $M$ is topologically a product, i.e.\ if $M$ is diffeomorphic to $\bbR \times T^{n-1}$.

As a special class of illustrative examples, consider 
the  Birmingham-Kottler metrics \cite{ChruBook, Birm}, with negative cosmological constant, in the $k =0$ (specifically, flat toroidal) case. In the usual coordinates these are given by, 
\beq
 M = (r_0, \infty) \times T^{n-1} \,, \qquad g_m = \frac{1}{r^2 - \frac{2m}{r^{n-2}}}dr^2  + r^2 h \,,
\eeq
where $(T^{n-1}, h)$ is the standard flat torus, and (up to a positive scale factor) $m$ is the mass. 

If  $m >0$ then  $r_0$ is the unique  root of 
 $\varphi(r) = r^2 - \frac{2m}{r^{n-2}}$.  Similar to the situation for the time slices of the standard Schwarzschild solution, 
$r = r_0$ is a coordinate singularity. As $r \to r_0$, the torus $T_r = \{r\} \times T^{n-1}$ limits to a totally geodesic flat torus.  $M$ can be continued
by, roughly speaking, reflecting across this torus; see  \cite{Brill}, \cite[Secs.\ 5.5, 6.3]{ChruBook}. The result is a complete time slice with topology $\bbR \times T^{n-1}$ and scalar curvature $R = -n(n-1)$. 

If instead $m < 0$, one can take $r_0 = 0$. In this case the metric $g_m$ extends a finite distance (from any fixed torus) to $r =0$, at which the curvature blows up in an orthonormal frame (even though the scalar curvature remains $-n(n-1)$); see \cite[Sec.\ 5.5.2]{ChruBook}.  In the special case $m= 0$ we obtain the locally hyperbolic metric, 
\beq\label{eq:lochyp}
g_0 = \frac{1}{r^2} dr^2 + r^2 h
\eeq
which after the coordinate change $u = \ln r$ gives,
$$
M = (-\8,\8) \times T^{n-1} \,, \qquad  g_0 =  du^2 + e^{2u} h  \,.
$$
We remark that this solution has been shown under various conditions to be the unique manifold with vanishing mass in the ALH setting; see \cite{ChruPRD, HuangJang}.   

Thus, we observe that for this class of Birmingham-Kottler metrics the hypotheses of 
Theorem~\ref{thm:pm1} are satisfied if and only if the mass is nonnegative, $m \ge 0$.
On the other hand, consider the Horowitz-Myers soliton \cite{HMsoliton}. It is a static spacetime solution to the vacuum Einstein equations with negative cosmological constant, which has negative mass. 
The time slices are asymptotically toroidal and, under suitable scaling, satisfy all the assumptions of Theorem \ref{thm:pm1}, except for the conditions on $N$, in particular the assumption that $N$ be noncompact. In the Horowitz-Myers case, $N$ is diffeomorphic to $T^{n-2} \times D^2$. Also, $N$  fails (just barely) to be asymptotically retractible.  If one removes from $M$ the `central torus', then the conditions on $N$ are satisfied, but now $M$ is no longer complete.

Regarding the dimension restriction in Theorem \ref{thm:pm1},  as in \cite[Thm.~1.1]{CGNP}, the proof makes use of the existence and regularity theory of marginally outer trapped surfaces (see  Section 2 for discussion), which restricts the dimension $n$ to the range $3 \le n \le 7$.  The additional restriction $n \ge 4$  is a technical condition arising in the proof of  \cite[Theorem 1.3]{CGNP}, which is used in the proof of Theorem \ref{thm:pm1}.  If one assumes that the {\it mass aspect function}  associated to the end $\calE$  has a sign, then Theorem \ref{thm:pm1} also holds  in dimension $n =3$.

In Section 2 we present some basic definitions and results used in the proof of Theorem \ref{thm:pm1}.  In Section 3 we give the proof of Theorem \ref{thm:pm1} and consider some additional results.

\section{Preliminaries}\label{prelim}

{\it Conformal compactification.}   Let $(M^n,g)$ be a smooth Riemannian manifold of dimension $n$.  Following the treatment in \cite{CGNP} (cf.\ \cite{GW}), we say that $M$ has an ALH toroidal end if there exists an $n$-dimensional submanifold with boundary $\mathcal{E}$, which is closed as a subset of $M$, such that, up to isometry,
\beq\label{eq:conf}
\calE = (0,x_0] \times T^{n-1} , \qquad g|_{\calE} = g_0  + x^{n-2} \kappa + o( x^{n-2})  \,,
\eeq
 where $g_0$ is the locally hyperbolic `reference metric',
 \beq
 g_0 = \frac{1}{x^2}[dx^2 +h]
 \eeq
 and $(T^{n-1}, h)$ is the standard flat torus. (Under the coordinate transformation $r = \frac{1}{x}$ we obtain \eqref{eq:lochyp}.)  As defined, $\calE$ is conformally compactifiable,  with defining function $x$ and with conformal infinity $\d_{\infty}\calE = \{0\} \times T^{n-1}$.  We assume that all terms in \eqref{eq:conf} are at least $C^2$,  and that the scalar curvature satisfies the mild decay condition 
 \cite[Eq.~4.1]{CGNP}.   We also assume that the error term satisfies the additional decay condition  \cite[Eq.~4.2]{CGNP}.  This is satisfied, for example, by the Birmingham-Kottler metrics and the Horowitz-Myers soliton.    
 
 The term $\kappa$ in \eqref{eq:conf} is a $(0,2)$ tensor on $T^{n-1}$, called the mass aspect tensor.  The {\it mass aspect function} $\mu$ is obtained by tracing $\kappa$ with respect to  $h$, $\mu = {\rm tr}_{h} \kappa$.  Then we define the mass $m = m(g)$ with respect to the end $\calE$,  as 
\beq\label{eq:mass}
m = \int_{T^{n-1}} \mu \ d\Omega_h  \,.
\eeq
Under a suitable normalization, the mass, so defined, is a geometric invariant of $g$.

With respect to a toroidal end $\calE$ of $M$, we say that  $N = \overline{M\setminus\calE}$ is {\it asymptotically retractible onto} $\d\cal{E}$ if $N$ admits an exhaustion by an increasing sequence of sets $N_i$ open in $N$, each of which is a retract onto $\d \cal{E}$.  
If $N$ is diffeomorphic to a product, $N \approx [0, \infty) \times T^{n-1}$, then it is trivially asymptotically retractible onto $\d\cal{E}$.  However $N$'s with much more complicated topology, e.g. by taking connected sums, can satisfy this condition.

\smallskip
\noindent
{\it Marginally outer trapped surfaces.}  The proof of Theorem \ref{thm:pm1} makes use of properties of marginally outer trapped surfaces (MOTS).  Here we review  basic definitions, and certain properties of MOTS.  

MOTS are defined with respect to initial data sets.  An initial data set $(M, g, K)$ consists of a  Riemannian manifold $(M, g)$  and a symmetric $(0, 2)$-tensor field $K$.  
The most basic physical application is when $M$ is a spacelike hypersurface in a spacetime 
$\mathcal{M}$, with induced metric $g$ and second fundamental form $K$. 

Associated to an initial data set $(M, g, K)$, are the \textsl{local energy density} $\mu$ and the \textsl{local current density} $J$ of $(M, g, K)$, given by 
\beq\label{eq:muj}
\mu = \frac{1}{2} \, \left( R - |K |^2 + (\text{tr} \, K)^2 \right) \quad \text{ and } \quad 
J = \text{div} \left(K - (\text{tr} \, K) \, g\right) \,,
\eeq
where $R$ is the scalar curvature of $(M, g)$.  When $(M, g, K)$ is an initial data set embedded in a spacetime $\mathcal{M}$ satisfying the Einstein equations, these quantities are obtained from expressions involving the energy-momentum tensor via the  Gauss-Codazzi equations.
$(M,g,K)$ is said to satisfy the \textsl{dominant energy condition} (DEC)  if 
\[
\mu \geq |\, J \, |.
\]

Let $\Sigma \subset M$ be a closed (compact without boundary) two-sided hypersurface in $M$, and let $\nu$ be a smooth unit normal field along $\S$ in $M$. By convention, we refer to $\nu$ as outward pointing.  The mean curvature of $\S$ in $M$ can then be expressed as,
\[
H = \text{div}_\Sigma \, \nu \,.
\]

We  define the null expansion scalars $\th^+$ and $\th^-$  along $\S$ as follows,
\[
\theta^+ =  \text{tr}_\Sigma (K) +H \qquad \text{ and } \qquad 
\theta^- = \text{tr}_\Sigma (K) - H \,.
\]
When $(M,g,K)$ is an initial data set in a spacetime $\mathcal{M}$, $\th^{\pm}$ may be expressed as
$$
\th^{\pm} = \div_{\S} \ell^{\pm} 
$$
where $ \ell^{\pm} = u \pm \nu$ are the future-directed null normal vector fields along $\S$. (Here, $u$ is the future directed unit normal to $M$.)  Thus, $ \ell^{\pm}$ measures the divergence of the outgoing/ingoing light rays from $\S$.

We say that $\Sigma$ is outer trapped if $\theta^+ < 0$ and weakly outer trapped if 
$\theta^+ \leq 0$.  If $\theta^+$ vanishes identically, $\th^+ \equiv 0$, 
$\S$ is called a marginally outer trapped surface (MOTS for short). 

Despite the absence of a variational characterization of MOTS like
that for minimal surfaces, MOTS have been shown to satisfy a number of properties
analogous to those of minimal surfaces in Riemannian geometry.  We will need to make use of the basic existence results for MOTS.

\begin{lem}[\cite{AndMet,Eic09,Eic10}]\label{thm:MOTSexist}
Let $(M,g,K)$ be an $n$-dimensional, $3\le n\le 7$, compact-with-boundary initial data set. Suppose that the boundary can be expressed as a disjoint union $\d M=\Sigma_{\inn}\cup\Sigma_{\out}$, where $\Sigma_{\inn}$ and $\Sigma_{\out}$ are nonempty unions of components of $\d M$ with $\theta^+\le 0$ on $\Sigma_{\inn}$ with respect to the normal pointing into $M$ and $\theta^+>0$ on $\Sigma_{\out}$ with respect to the normal pointing out of $M$. Then there is an outermost separating MOTS in $(M,g,K)$ that is homologous to  $\Sigma_{\out}$.
\end{lem}

A considerable amount of work has been done on rigidity aspects of MOTS; see e.g. \cite{EGM,GM18, GMTrans}, to name just a few.  Here, we will make use of a local rigidity result  in \cite{motsv4}, which, together with certain arguments used to prove it, has played a role in some of these rigidity papers.

Let $\S$ be a separating MOTS in an initial data set $(M,g,K)$, with `outward' normal 
$\nu$.  Then $\Sigma$ is said to be {\it weakly outermost} if there is no closed embedded surface outside of $\S$ (side to which $\nu$ points) with $\theta^+ < 0$ that is homologous to $\S$. 

\begin{lem}[\cite{motsv4}]\label{locrigid}
Let $(V^n,h,K)$, $n \ge 3$, be an initial data set satisfying the DEC, $\mu \ge |J|$. 
Suppose $\S^{n-1}$ is a weakly outermost MOTS in  $V^n$
that does not admit a metric of positive scalar curvature.  Then there exists
an outer neighborhood $U \approx [0,\e) \times \S$ of $\S$ in $V$ such that each slice
$\S_t = \{t\} \times \S$, $t \in [0,\e)$ is a MOTS.  
\end{lem}
There are further rigidity type conclusions in this theorem, but the existence of this local  foliation by MOTS is sufficient for our purposes.  

Lemma \ref{locrigid} will be used in conjunction with a result from \cite{EGM}.  
An orientable, closed manifold $N^m$ of dimension $m$ is said to satisfy the \textsl{cohomology condition} if there are classes $\omega_1,\, \ldots,\, \omega_m\in H^1(N,\bbZ)$ whose cup product
\begin{align*} 
\omega_1\smile\cdots\smile\omega_m\in H^m(N,\bbZ)
\end{align*}
is non-zero. 
The $m$-torus $T^m$ is a prototypical example satisfying the cohomology condition. As shown in \cite[Theorem~2.28]{Lee:book},  such a manifold $N^m$, $3 \le m \le 7$, has a component that does not admit a metric of positive scalar curvature (cf. \cite{SY2017}).   We will use this fact, together with the following.

\begin{lem}[{\cite[Lemma 3.1]{EGM}}]\label{lem:cohomology}
Let $M$ be an orientable, compact $n$-dimensional, $n \geq 3$, manifold with boundary. 
Let $\Sigma_0$ be a component of $\partial M$ satisfying the cohomology condition, and suppose there exists a retract of $M$ onto $\S_0$.   Then every closed, embedded hypersurface $\Sigma \subset M$ homologous to $\Sigma_0$ satisfies the cohomology condition.
\end{lem}

We remark that in \cite{EGM}, a condition slightly more general than assuming the existence of a retract $F:M \to \S$ is assumed.  

\section{Proof of Theorem \ref{thm:pm1} and further comments}

\proof[Proof of Theorem \ref{thm:pm1}]    As  in the proof of \cite[Thm. 1.1]{CGNP}, we introduce a certain initial data set and make use of properties of MOTS.   The initial data set used here is motivated by arguments in \cite[Sec. 6]{LLU23} and \cite{TinYau} (cf.\  \cite[Sec. 3.3]{GMEH}). 

Suppose, by contradiction, that the mass, as given in \eqref{eq:mass}, is negative, 
$m <0$.
Let $\calE$ be the asymptotically toroidal end in $M$, with notation as in Section \ref{prelim}.  For $x_1 \in (0, x_0)$, let $\S_1$ be the toroidal slice $x = x_1$.  As computed in \cite[Section 6]{CGNP}, the mean curvature $H_1$ of $\S_1$, with respect to 
the `outward' normal $\nu$ (i.e. the normal pointing towards the conformal boundary) of 
$\S_1$ is given by,
\beq\label{eq:mean}
H_1 = (n-1) - \frac12 n \mu x^n + o(x^n)  \,,
\eeq
where, recall, $\mu =  {\rm tr}_{h} \kappa$ is the mass aspect function.  Now, by applying 
Theorem 1.3 in \cite{CGNP}, we may assume, without loss of generality, that $\mu$ is constant, and, moreover, by \eqref{eq:mass} must be negative, $\mu < 0$.  By choosing $x_1$ sufficiently small, equation \eqref{eq:mean} then implies 
\beq\label{h1}
H_1 > \l (n-1)   \quad \text{for some  number $\l > 1$.}
\eeq

Now consider the initial data set $(M_0,g ,\hat{K} = K  - \frac{h}{n-1}g)$, where $K = -\l g$. Here  $M_0$ is an open submanifold of $M$, and $h$ is a smooth real-valued function on $M_0$,
  which we describe momentarily after identifying several open sets in $M$.  For $x_2 \in (x_1,x_0)$, let $U_2 = \{(x,x^a) \in \calE: x <  x_2\}$, and let $U_1 =  {\rm int}\ \calE = \{(x,x^a) \in \calE: x <  x_0\}$. 
Finally, let $U_0 = N_0 \cup U_1 = N_0 \cup \calE$, 
where, for any $D_0 > 0$,  $N_0 = \{p \in M: d(p, \d U_1) < D_0\}$.  By completeness, 
$\overline{U_0\setminus \calE}$ and $\d U_0$ are compact. 
 Also, set $D_1 = d(\d U_1, \d U_2)$. (Here  $d$ is the Riemannian distance with respect 
 to~$g$.)

\newpage
With $M_0$ and $h \in C^{\infty}(M_0)$  constructed as in \cite[Sec.\ 6]{LLU23} (cf. \cite[Sec.\ 3]{LUY}),  the following holds for arbitrarily small $\varepsilon > 0$:

\vs{-.1in}
\ben
\item[(i)] $U_1 \subset M_0 \subset U_0$, with $\d M_0$  a smooth compact hypersurface.

\vs{-.05in}
\item[(ii)]   $h \ge 0$, with $h = 0$ on $U_2$,  and $h \to \infty$ uniformly on approach to 
$\d M_0$.

\vs{-.05in}
\item[(iii)]   Further,
$$
|\D h| \le \frac{2(1+ \vep)}{(D_0- 2\vep)D_1}  \quad  \text{on}  \quad \overline{U}_1 \setminus  U_2 \,,   
$$

\vs{-.15in}
\item[(iv)]   and,
$$
|\D h| \le \frac{1 + \vep}{2} h^2 < \frac{1}{2}\frac{n}{n-1}h^2  \quad \text{on} \quad M_0 \setminus U_1  \,,
$$
by taking $\vep < \frac{1}{n-1}$.
\een
 As we now indicate, for $M_0$ and $h$ so constructed, the initial data set $(M_0,g ,\hat{K})$
satisfies the DEC, provided, using completeness, one takes $D_0$  sufficiently large.  

By a straightforward computation  \cite{LLU23, TinYau} (see also \cite{GMEH}), we have
\begin{align*}
\hat{\mu}-|\hat{J}|&=\mu+\frac{1}{2}\left(\frac{n}{n-1}h^2- 2h\ {\rm tr}_g K \right)-|J+Dh|\\
&\ge\mu-|J|+\frac{1}{2}\frac{n}{n-1}h^2 +n\l h -|\D h|  
\end{align*}
where we have used $K = - \l g$.  Further, with respect to the initial data set 
$(M,g,K = -\l g)$, using \eqref{eq:muj} and the scalar curvature assumption gives
$$
\mu - |J| = \frac12 (R +  \l^2 n(n-1)) \ge \frac12 (\l^2 - 1)n(n-1)  \, ,
$$
and hence,
$$
\hat{\mu}-|\hat{J}| \ge \frac12 (\l^2 - 1)n(n-1) + n\l h +\frac{1}{2}\frac{n}{n-1}h^2  -|\D h| \,,
$$
where, recall, $\l > 1$.  By taking $D_0$ sufficiently large (as we can by completeness), it now follows easily from (ii)-(iv) that $\hat{\mu}-|\hat{J}| \ge 0$ on $M_0$, and hence that the initial data set  $(M_0,g ,\hat{K})$ satisfies the DEC.

Consider the null expansion of $\th^+_1$ of $\S_1$ with respect to $(M_0,g,\hat{K})$.  Since 
$h  = 0$ on $U_2$, $\hat{K} = K =  -\l g$ on $U_2$.  From \eqref{h1} this implies,
$$
\th^+_1 = \text{tr}_\Sigma (\hat{K}) +H_1 > -\l(n-1) + \l(n-1) = 0.
$$

Now consider a hypersurface $\S_0 \subset M_0$ parallel to $\d M_0$, with respect to a normal neighborhood of $\d M_0$.
The null expansion of 
$\S_0$ is given by,
$$
\th^+_0 = \text{tr}_{\S_0} (\hat{K}) +H_0  = -\l (n-1) - h|_{\S_0}  +H_0  \,. 
$$
Since $h$ becomes arbitrarily large near $\d M_0$,  it follows  that $\th^+_0  <  0$ if 
$\S_0$ is taken sufficiently close to $\d  M_0$.  
Now let $W$ be the compact region bounded by $\S_0$ and $\S_1$.  $W$ satisfies the barrier conditions of Lemma \ref{thm:MOTSexist}, and hence there exists an outermost MOTS $\S_*$ in $W$ homologous to $\S_1$.  

Choose an $N_i$ containing $W \cap N$.  Using that $N_i$ retracts onto $\d \calE = \{x = x_0\}$ and that the product region: $x_1 \le x \le x_0$ retracts onto $\S_1 = \{x = x_1\}$, one can see that $W$ retracts onto $\S_1$.
Since $\S_1 \approx T^{n-1}$
satisfies the cohomology condition, it follows from Lemma~\ref{lem:cohomology} that 
 $\S_*$ satisfies the cohomology condition. Hence, by the comments just before the lemma, some component $\S'_*$ of $\S_*$ does not carry a metric of positive scalar curvature.  Moreover,  $\S'_*$ must be weakly outermost.  Otherwise, by replacing $\S'_*$ by an outer trapped surface, one could apply Lemma \ref{thm:MOTSexist} to obtain a MOTS with at least some component strictly outside of
$\S_*$, thereby contradicting that $\S_*$ is outermost. 
But then by Lemma~\ref{locrigid}, an outer neighborhood of $\S'_*$ is foliated by MOTS, which again contradicts that $\S_*$ is outermost.\qed

\smallskip
\noindent
{\it Remark.}  One may also reach a contradiction by applying the global rigidity result, Theorem 1.2 in \cite{EGM}, to the initial data set $(W,g, - \hat{K})$.
However, as the proof of that theorem depends crucially on the lemmas in Section \ref{prelim}, invoking these lemmas directly leads to a more transparent proof of Theorem \ref{thm:pm1}.

\smallskip

In the absence of the completeness assumption, one still has the following `quantitative shielding' result (cf.\  \cite[Thm. 1.1]{LLU23}, \cite[Thm. 1.1]{TinYau}).

\begin{thm}\label{thm:pm2} 
Let $(M^n,g)$, $4\le n\le 7$, be an ALH manifold with conformally compactifiable  toroidal end
$\calE$.  Let $U_0$, $U_1$, and $U_2$ be neighborhoods of $\calE$ such that  
$\overline{U_2}\subset U_1$, $\overline{U_1}\subset U_0$, and $\overline {U_0 \setminus \mathcal E}$ is compact,
and let 
\[D_0={\rm dist}_g(\partial U_0, U_1)\quad\text{and}\quad
D_1 ={\rm dist}_g(\partial U_1, U_2).\] 
Suppose the following hold.

\begin{enumerate}
    \item[(1)]  There is a retraction of $U_0$ onto $\d \calE$,
     
\vs{-.05in}    
    \item [(2)]$R_g\ge -n(n-1)$ on $U_0$, and 
    
\vs{-.05in}       
    \item[(3)] \label{scalar_bound} the scalar curvature satisfies the strictness condition,
    \begin{equation}\label{eq:strictness}
R_g +n(n-1) >\frac{4}{D_0 D_1}\quad \text{on}\quad\overline{U_1}\setminus U_2 \,.
    \end{equation}
\end{enumerate}
Then the mass of the end $\calE$ is nonnegative.
\end{thm}

The proof follows along lines similar to that of Theorem \ref{thm:pm1}.   One again assumes, by contradiction, that the mass is negative, and considers as before  the initial data set 
$(M_0,g, \hat{K})$, except now one can take $K = -g$ (i.e.\ take $\l = 1$).  In checking that 
$(M_0,g, \hat{K})$ satisfies the DEC, the strict scalar curvature condition \eqref{eq:strictness} replaces 
the need to take $\l > 1$ and the need to make $D_0$ sufficiently large.  

A version of Theorem \ref{thm:pm2} for manifolds with boundary, when the mean curvature inequality $H \le n-1$ is not satisfied can be obtained along the lines of \cite[Thm.~1.7]{LLU23} and \cite[Thm. 1.2]{TinYau}.   Theorem 3.6 in \cite{GMEH} can also be modified to obtain a 
manifold-with-boundary result in the present setting. 

Finally, for a positive mass result for complete spin manifolds with toroidal end, see \cite{ChruMTod}.

\bigskip
\noindent
\textsc{Acknowledgements.} We thank Piotr Chru\'sciel and Abrãao Mendes for helpful comments.

\bibliographystyle{amsplain}
\bibliography{torus.bib}

\end{document}